\documentclass{elsarticle}
\usepackage{amsfonts,amscd,amssymb,amsmath,amsthm}
\usepackage{graphicx,mathrsfs,appendix}
\usepackage{centernot,caption,subcaption,color,url}
\usepackage[T1]{fontenc}
\usepackage{ctable}

\newtheorem{theorem}{Theorem}[section]
\newtheorem{lemma}[theorem]{Lemma}

\theoremstyle{definition}

\newcommand{\A}{\mathcal{A}}

\newcommand{\N}{\mathbb N}

\begin{document}
\title{First-order logic of uniform attachment random graphs with a given degree\tnoteref{t1}}
\tnotetext[t1]{Present work was funded by RFBR, project number 19-31-60021.}
\author[1]{Y.A. Malyshkin}
\ead{yury.malyshkin@mail.ru}


\address[1]{Moscow Institute of Physics and Technology//Tver State University}

\begin{abstract}
In this paper, we prove the first-order convergence law for the uniform attachment random graph with almost all vertices having the same degree. In the considered model, vertices and edges are introduced recursively: at time $m+1$ we start with a complete graph on $m+1$ vertices. At step $n+1$ the vertex $n+1$ is introduced together with $m$ edges joining the new vertex with $m$ vertices chosen uniformly from those vertices of $1,\ldots,n$, whom degree is less then $d=2m$. To prove the law, we describe the dynamics of the logical equivalence class of the random graph using Markov chains. The convergence law follows from the existence of a limit distribution of the considered Markov chain.
\end{abstract}


\maketitle

\section{Introduction}


In the present paper, we proof of the FO convergence law for uniform attachment random graphs with most vertices having a given degree using finite Markov chains.

FO sentences about graphs could include the following symbols: variables $x,y,x_1,\ldots$ (which represent vertices), logical connectives $\wedge,\vee,\neg,\Rightarrow,\Leftrightarrow$, two relational symbols (between variables) $\sim$ (adjacency) and $=$ (equality) , brackets and quantifiers $\exists,\forall$  (see the formal definition in, e.g.,~\cite{Libkin}). The sequence $G_n$ of random graphs obeys the FO convergence law, if, for every FO sentence $\varphi$, $\Pr(\mathcal{G}_n\models\varphi)$ converges as $n\to\infty$. If the limit is eighter $0$ or $1$ for any formula, $G_n$ obeys the zero-one law. If $G_n$ obeys the zero-one law, it is trivial in terms of the FO logic in the sense that all properties are trivial on a typical large enough graph. 

The FO logical laws usually proven using Ehrenfeucht-Fra\"{\i}ss\'{e} pebble game (see, e.g., \cite[Chapter 11.2]{Libkin}). The connection between FO logic and the Ehrenfeucht-Fra\"{\i}ss\'{e} game is described in the following result.
\begin{theorem}
\label{thm:Ehren}
Duplicator wins the $\gamma$-pebble game on $G$ and $H$ in $R$ rounds if and only if, for every FO sentence $\varphi$ with at most $\gamma$ variables and quantifier depth at most $R$, either $\varphi$ is true on both $G$ and $H$ or it is false on both graphs.
\end{theorem}
In particular, if for any $\epsilon>0$ one could define a finite number of classes $\A_k$, $k=1,...,K$, of graphs such that Duplicator wins the pebble game on graphs of the same class and $\Pr(G_n \in \A_k)\to p_k$ for all $k=1,...,K$ with $\sum_{k=1}^{K}p_k>1-\epsilon$, then $G_n$ obeys the FO convergence law.

Logic limit laws were studied on many different models, such as the binomial random graph (\cite{Spencer_Ehren,Shelah}), random regular graphs (\cite{Haber}), attachment models (\cite{MZ21,M22,MZ22}), etc. (\cite{Muller,Strange,Winkler}). In the present article, we are interested in recursive random graph models. One of the main arguments towards $G_n$ not satisfying the zero-one law for such models is the existence of rare subgraphs which appear with diminishing probability (and with high probability does not appear after some moment, see, e.g. \cite{MZ21}). In this case, even if $G_n$ obeys the FO convergence law and not the zero-one law, it still could be asymptotically trivial in terms of some classes of sentences of the FO logic in a sense that the correctness of the property on a graph does not change after some moment (see, e.g., \cite{M22,MZ22}). In such a case, the division on the classes $\A_k$, $k=1,...,K$, is based on subgraph of $G_n$ on first $N$ vertices for all $n>N$ ($N$ depends on $\epsilon$). The more difficult case is when the correctness of the property changes infinitely many times during the course of the (graph) process. Such behavior is somewhat similar to the behavior of a Markov chain, which changes its state infinitely many times but still could have a limiting probability distribution. One of the main purposes of the present paper is to showcase such a connection between a graph obeying the FO convergence law and the existence of the limiting probability distribution for related Markov chains (in our case such a chain would be finite). 

Let us give a formal description of our model. Fix $m\in\N$, $m>1$, and let $d=2m$. We start with a complete graph $G_{m+1}$ on $m+1$ vertices. Then on each step, we add a new vertex and draw $m$ edges from it to different vertices, chosen uniformly among vertices with a degree less than $d$. Note that the total degree (the sum of degrees of all vertices) of the graph $G_n$ would be equal to $m(m-1)+2m(n-m)=dn-m(m+1)$. In particular, the number of vertices of degree $d$ in $G_n$ is between $n-m(m+1)$ and $n-(m+1)$. Hence it is always possible to draw $m$ edges from a new vertex to different vertices (of a degree less than $d$). The model for $d>2m$ was considered in \cite{MalZhu22}, where a similar result was proven using Markov chains with infinitely many states, as well as the stochastic approximation technique.

The resulting graph is somewhat similar to a $d$-regular graph (i.e. graph with all vertices having degree $d$, either number of vertices or the degree should be even for such a graph to exist). Note that regular graphs could not be dynamic since if we draw edges to vertices of a regular graph it would no longer be regular. Hence, regular graphs are built separately for each $n$ (where $n$ is the number of vertices). Our model could be considered as a way to build a graph with properties close to a regular graph through the dynamic procedure.

Let us formulate our main result.

\begin{theorem}
\label{thm:main}
$G_n$ obeys FO convergence law.
\end{theorem}

\section{Configuration of open vertices}

For a graph $G_n$, consider its subgraph $T_n$, obtained in the following way. First, consider the subgraph of $G_n$ on vertices that have degrees less than $d$, and for each vertex add to it the number of leaves equal to the difference between vertex degree in $G_n$ and such a subgraph.  

Note that the set of all possible subgraphs obtained in such a way is finite. Also, $T_{n+1}$ depends only on $T_n$ and does not depend on $n$. As result, $T_n,$ $n\in\N$ form a Markov chain (see, e.g., \cite[Chapter 6]{GS01} for more details on Markov chains and corresponding terminology) with finite number of states (which corresponds to different types of subgraphs $T_n$).

\begin{lemma}
\label{lem:tree_state}
A state that consists of $m+1$ isolated vertices with $m$ leaves is achievable from any state in $m+1$ steps (we would call such a state a forest state).
\end{lemma} 
\begin{proof}
Since the total degree of $G_n$ equals to $2mn-m(m+1)$, it would take $m+1$ steps of drawing $m$ edges to $G_n$ from a new vertices to make all vertices of $G_n$ to have degrees equal to $d=2m$.
Let us consider the following procedure. At each step, we would choose $m$ vertices among vertices of $G_{n}$ with the smallest degrees and draw edges to them. Since before $k$-th step the total degree of vertices of $[n]$ ($[n]$ denotes first $n$ vertices, i.e. all vertices of $G_n$) would be equal to $2mn-m(m+1-(k-1))$, there could not be more then $m$ vertices with degree at most $d-m+k-1$, which means if the minimum degree before $k$-th step does not exceed $d-m+k-1$, it would increase by $1$. Hence, the minimum degree after $k$-th step would be at least $d-m+k$, and there would be at least $m$ vertices with degrees less than $d$, so we could make the next step.
As result, we would get $m+1$ vertices of degree $m$ that are connected only to vertices of $[n]$, which have degrees equal to $d$.
\end{proof}
From this lemma, it follows that we could return from the forest state to itself in both $m+1$ and $m+2$ steps (we could make a single step and then apply this lemma to the resulting state). Therefore, all states achievable from the forest state form an irreducible persistent Markov chain, and, hence, there exists a limiting probability distribution.

Define $U^a_n(t)$ as $a$-neighborhood of vertices of $G_n$ with degrees less then $d$ (at moment $n$) at time $t\geq n$. There is a finite number of possible neighborhoods. Transitions between $U^a_n(n)$ and $U^a_{n+1}(n+1)$ depends only on state of $U^a_n(n)$, while transition between  $U^a_n(t)$ and $U^a_n(t+1)$ depends on $U^a_n(t)$ and $U^a_t(t)$.
Therefore, $U^a_n(n)$ form a finite Markov chain over $n$ and a pair $(U^a_n(t),U^a_t(t))$ form a finite Markov chain over $t$.

For the chain $U^a_n(n)$ define the following basic state. Consider $a+1$ consequent groups $A_1,...,A_{m+1}$ of $m$ vertices, such that every vertex of $A_i$ connected with and only with every vertices of $A_{i-1}$ and $A_{i+1}$. Note that this state is achievable from any state in $m(a+1)$ steps, by repeating procedure described in Lemma~\ref{lem:tree_state} $a+1$ times. Therefore, this chain is irreducible and all states achievable from the basic state (and only them) are persistent, and, hence, there exists limiting probability distribution. Therefore, we get the following result.
\begin{lemma}
\label{lem:neighbourhood_type}
Let a graph $H$ belong to a set $\mathcal{U}^{a}$ of all possible neighborhoods $U^a_n(n)$, achievable from the basic state.
Then 
$$\Pr\left(U^a_{n}(n)=H\right)\to c_H$$
almost surely, and
$$\sum_{H\in\mathcal{U}^a}c_H=1.$$
\end{lemma} 

\section{Subgraphs on old vertices}

In this section, we study $a$-neighborhoods of vertices that contain only vertices of degree $d$ (in graphs $G_n$, for large enough $n$, once such a neighborhood is achieved on a vertex it does not change afterward). We would call such neighborhoods complete. Note that only a finite number of such neighborhoods could be achieved. Moreover, for each possible configuration of $U^a_n(n)$ the probabilities of obtaining any achievable set of complete neighborhoods on vertices of $T_n$ in bounded (by some constant $C$) number of steps depends only on graph $U^a_n(n)$ (and does not depend on $n$). Due to Lemma~\ref{lem:neighbourhood_type} the probability for  $U^a_n(n)$ to have a given configuration separated from $0$ for large enough $n$. Therefore for a vertex $n$ probabilities that its $a$-neighborhood at time $n+C$ would be complete and of a given achievable type is separated from $0$. Hence, we get the following result.
\begin{lemma}
\label{lem:vertex_neighbourhoods}
For any $k$ and any achievable type of complete $a$-neighborhood with a high probability, there are at least $k$ vertices with disjoint $a$-neighborhoods of that type in $G_n$.
\end{lemma} 

\section{Convergence law}

Fix $R\in\N$. Let $a=3^R$.
In this section we provide division on classes $\A_k$ of graphs and prove existence of the winning strategy in $R$ rounds  for a pair of graphs $G_{n_1}$, $G_{n_2}$ within the same class (for large enough $n_1,n_2$). We would consider a division based on the type of the graph $U^a_n(n)$ and the type of graph on initial vertices. The configuration on $m+1$ initial vertices is not achievable again during the process. Therefore, to insure that configuration of $U^a_n(n)$ is achievable, we need to make sure that the first $m+1$ vertices do not belong to $U^a_n(n)$. Since the number of vertices with a degree less than $d$ does not exceed $m(m+1)$, the probability to increase the degree of a given vertex is at least $\frac{1}{m+1}$. Hence, with high probability $a$-neighborhoods (we denote their union as $W^a(n)$) of the first $m+1$ vertices in $G_n$ contains only vertices of degree $d$ (by standard large deviation estimates such probability is at least $1-Ce^{-cn}$ for some constants $c, C>0$), so we get
\begin{lemma}
\label{lem:first_vertices}
For any $n_0$ with high probability all degrees of vertices from $[n_0]$ has degree $d$ in $G_n$. 
\end{lemma}

Fix $\epsilon>0$. Let $N$ be such that with probability at least $1-\epsilon$, degrees of all vertices of $W^{2a}(N)$ in $G_N$ equal to $d$ (the same would be then true for all $n\geq N$).
There is a finite number $M$ of pairs of types of $W^{2a}(n)$ and $U^a_n(n)$. Let classes $\A_k,$ $k=1,\ldots,M,$ be defined by pairs $(W^{2a}(n),U^a_n(n))$, such that degrees of all vertices of $W^{2a}(N)$ in $G_N$ equal to $d$. Let us define the following properties of graphs $G_{n_1},G_{n_2}$.
\begin{itemize}
\item[${\sf Q1}$] $G_{n_1}$ and $G_{n_2}$ belong to the same class $\A_k$.
\item[${\sf Q2}$] For any achievable type of complete $a$-neighborhood of a vertex there are at least $R$ vertices with non-intersecting $a$-neighborhoods of that type in $G_n$ that does not intersect with $W^{2a}(n)$ and $U^a_n(n)$, $n=n_1,n_2$.
\end{itemize}
Note that the probability that for all $n>N$ graph $G_n$ belongs to one of the classes $\A_k$ is at least $1-\epsilon$.

\begin{lemma}
\label{lem:game}
If graphs $G_{n_1}$, $G_{n_2}$ satisfy properties $Q1,Q2$, then Duplicator has a winning strategy on them.
\end{lemma}

\begin{proof}
Let us consider the following strategy. For a vertex $v$ and $r\in\N$ let $B_r(v)$ be its neighborhood of radius $r$.
Let Spoiler be putting pebbles $x_1,...,x_R$ and duplicator putting $y_1,...,y_R$. We omit a reference to a graph in the notation for these balls -- each time we use the notation, the host graph would be clear from the context. We need to make a strategy such that on each step subgraphs of $G_{n_1}$ and $G_{n_2}$ on pebbles are isomorphic. We will build the strategy by induction over $i$.
Let assume that $B_{2^{R-j+1}}(x_j)$ and $B_{2^{R-j+1}}(y_j)$ are the same (i.e. they isomorphic and keep correspondence between pebbles) for $j< i$. 

1. If $d(x_i,[m+1])<2^{R-i+1}$ then we put $y_i=x_i$. Note that their neighborhoods belong to $W^{a}(n)$ and, hence, are the same.

2. If $x_i$ belongs to one of $U^{2^{R-i+1}}_n(n)$, $n=n_1,n_2$, (without loss of generality assume it is $U^{2^{R-i+1}}_{n_1}(n_1)$), it's neighborhood belongs to $U^a_{n_1}(n_1)$ and since $U^a_{n_1}(n_1)$ and $U^a_{n_2}(n_2)$ are the same we could choose $y_i$ in $U^a_{n_2}(n_2)$ that corresponds to $x_i$ such that their neighborhood would be the same. 

3. If $x_i$ does not belong to either one of $W^{2^{R-i+1}}(n)$, $U^{2^{R-i+1}}_n(n)$, $n=n_1,n_2$, our goal is to choose $y_i$ in a way that its $2^{R-i+1}$ neighborhood would be exactly the same as of $x_i$. If there are pebbles in $B_{2^{R-i+1}}(x_i)$ (let $j$ be the lowest index of such a pebble), then $B_{2^{R-i+1}}(x_i)$ belongs to $B_{2^{R-j+1}}(x_j)$ (or $B_{2^{R-j+1}}(y_j)$), and, therefore, $x_i$ corresponds to the vertex in $B_{2^{R-j+1}}(y_j)$ (or $B_{2^{R-j+1}}(x_j)$), which we put as $y_i$. 
Now consider the case when there are no pebbles in $B_{2^{R-i+1}}(x_i)$.
If it does not belong to the neighborhood of one of the previous pebbles, there are at least $R-i+1$ vertices with exactly the same $a$-neighborhood (that does not interact with previously chosen vertices), and we choose any of them as $y_i$. If it belongs to the neighborhood of one of the previous pebbles, then it belongs along with its neighborhood to a wider $a$-neighborhood of one of the previous pebbles and hence corresponds to a vertex $y_i$ in the same $a$-neighborhood in the other graph.

\end{proof}

Now Theorem~\ref{thm:main} follows from
\begin{lemma}
\label{lem:properties}
For any $R\in\N$ and any $\epsilon>0$ there is $N\in\N$, and numbers $p_i>0$, $i\in[M]$, $\sum_{i=1}^{M}p_i=1-\epsilon$, such  that
\begin{itemize}
\item For any $\epsilon>0$ there is $N\in\N$, such that with probability at least $1-\epsilon$ for all $n_1>n_2>N$ the pair $(G^1,G^2)$ has the property ${\sf Q2}$;
\item for every $i\in[M]$, $\lim_{n\to\infty}\Pr(G_n\in\mathcal{A}_i)=p_i$.
\end{itemize}
\end{lemma}
\begin{proof}
The first part follows from Lemma~\ref{lem:vertex_neighbourhoods} and the fact that the number of complete neighborhoods of the given type in $G_n$ is non-decreasing over $n$.

The second part follows from Lemma~\ref{lem:first_vertices} and Lemma~\ref{lem:neighbourhood_type}.
\end{proof}

\section*{Acknowledgements.}
The study was funded by RFBR, project number 19-31-60021. The author is grateful to Maksim Zhukovskii for helpful discussions.

\end{document}